\documentclass{amsart} 
\usepackage{amsmath}
\usepackage{amssymb,latexsym}
\usepackage{setspace}

\usepackage{sseq}
\usepackage{comment}
\usepackage[all]{xy}
\newtheorem{thm}{Theorem}[section]
\newtheorem{lem}[thm]{Lemma}
\newtheorem{cor}[thm]{Corollary}

\newtheorem{prop}[thm]{Proposition}

\newtheorem{assump}[thm]{Assumption}
\theoremstyle{definition}
\newtheorem{defin}[thm]{Definition}

\newcommand{\lra}{\longrightarrow}

\newcommand{\ca}{\mathcal{A}}
\DeclareMathOperator{\id}{\rm id}

\DeclareMathOperator{\colim}{\rm colim}
\def\co{\colon\thinspace}
\DeclareMathOperator{\Ring}{\rm HoRing}
\DeclareMathOperator{\eRing}{\rm E_2-Ring}

\usepackage[colorlinks=true, linkcolor=blue, menucolor=blue]{hyperref}
\def\t{\mathbb{T}}
\def\c{\mathbb{C}}

\begin{document}
\title{The Segal Conjecture for topological Hochschild homology of Ravenel spectra}        

\author{Gabriel Angelini-Knoll}             
\email{gak	@math.fu-berlin.de}       
\address{Department of Mathematics \\ 
     Freie Universit\"at Berlin\\
         Germany\\
         Berlin\\
          Arnimalee 7, 14195}

\author{J.D. Quigley}            
\email{jdq27@cornell.edu}       
\address{Department of Mathematics\\ 	  
         Cornell University\\
   		U.S.A.\\
        Ithaca, NY \\
        580 Malott Hall, 14853}

\begin{abstract}
In the 1980's, Ravenel introduced sequences of spectra $X(n)$ and $T(n)$ which played an important role in the proof of the Nilpotence Theorem of Devinatz--Hopkins--Smith. In the present paper, we solve the homotopy limit problem for topological Hochschild homology of $X(n)$, which is a generalized version of the Segal Conjecture for the cyclic groups of prime order. This result is the first step towards computing the algebraic K-theory of $X(n)$ using trace methods, which approximates the algebraic K-theory of the sphere spectrum in a precise sense. We solve the homotopy limit problem for topological Hochschild homology of $T(n)$ under the assumption that the canonical map $T(n)\to BP$ of homotopy commutative ring spectra can be rigidified to map of $E_2$ ring spectra. We show that the obstruction to our assumption holding can be described in terms of an explicit class in an Atiyah-Hirzebruch spectral sequence. 
\end{abstract}



         
\maketitle

\tableofcontents

\section{Introduction}  
In the 1970's, Segal conjectured that after completion at the augmentation ideal, the Burnside ring of a finite group $G$ and the cohomotopy of $BG$ agree \cite{Ada82}. This conjecture inspired an outpouring of exciting research in the 1970's and early 1980's, leading to the resolution of the Segal conjecture for any finite group by Carlsson in \cite{Car84}. Let $p$ be a prime, let $C_p$ denote the cyclic group of order $p$, and let $S$ denote the $C_p$ equivariant sphere spectrum. The Segal Conjecture for $C_p$ may then be stated more generally as the question of whether the Tate-valued Frobenius map
\[ \varphi_p \colon \thinspace S\lra S^{tC_{p}}, \]
using the terminology of Nikolaus--Scholze \cite{NS18}, is an equivalence after $p$-adic completion. This version of the conjecture was resolved by Lin when $p=2$ and by Gunawardena when $p>2$ using an algebraic construction called the Singer construction \cite{LDMA80,AGM85}.

Given an $E_1$ ring spectrum $R$, the topological Hochschild homology of $R$, denoted $THH(R)$, has a canonical $\mathbb{T}$-action, where $\mathbb{T}\subset \mathbb{C}$ is the circle group, and consequently $THH(R)$ has a canonical $C_p$-action by restriction to the $p$-th roots of unity in $\mathbb{T}$. 
The Segal Conjecture for the group $C_p$ can then be rephrased as the question of whether the Tate-valued Frobenius map
\[ \varphi_p \colon \thinspace THH(S) \longrightarrow THH(S)^{tC_{p}},\] 
 is an equivalence after $p$-adic completion. Since there is a $C_p$ equivariant equivalence $THH(S)\overset{\simeq}{\to}S$ induced by tensoring with the $C_p$ equivariant collapse map $S^1\to *$ in the category of commutative ring spectra, this indeed implies the original Segal conjecture. 
 More generally, for a $E_1$ ring spectrum $R$, one may ask whether the Tate-valued Frobenius map
\begin{align}\label{eq1} \varphi_p \colon \thinspace THH(R) \longrightarrow THH(R)^{tC_{p}} \end{align} 
is an equivalence after $p$-adic completion. We therefore say that the Segal Conjecture for topological Hochschild homology of $R$ holds if the Tate-valued Frobenius map is an equivalence after $p$-adic completion, following Lun{\o}e-Nielsen--Rognes \cite{LNR11}. 

Notably, the Segal conjecture for topological Hochschild homology of $MU$ holds by \cite{LNR11}. The Segal conjecture for topological Hochschild homology of the Adams summand mod $(p,v_1)$ does not hold on the nose, but the Tate-valued Frobenius map induces an isomorphism in mod $(p,v_1)$ homotopy groups in sufficiently high degrees. This result was a crucial step in the calculation of mod $(p,v_1)$ topological cyclic homology of the Adams summand by Ausoni--Rognes \cite{AR02}. 

In the present paper, we prove the Segal Conjecture for topological Hochschild homology of the Ravenel spectra $X(k)$ holds for all $k\ge 1$. The spectra $X(k)$ are $E_2$ ring spectra that filter between the sphere spectrum and complex cobordism, in the sense that there is a sequence 
\[ S\to X(1)\to X(2) \to \dots \to X(k)\to X(k+1)\to \dots \to X(\infty)=MU\]
of maps of $E_2$ ring spectra. We refer to the maps
appearing in this sequence as the canonical maps $X(k)\to MU$ throughout. We also prove the Segal Conjecture for topological Hochschild homology of the Ravenel spectra $T(n)$ holds for all $n\ge 0$ under the hypothesis that the canonical maps $T(n)\to BP$ are map of $E_2$ ring spectra.\footnotemark \footnotetext{The notation we use for the Ravenel spectra $T(n)$ is the notation from \cite{Rav86}. We warn the reader that the same notation is used for the $v_n$ telescope of a type $n$ spectrum in \cite{Kuh07}.}
The Ravenel spectra $T(n)$ are known to split off of the $p$-localization of $X(p^n)$ in the same way that $BP$ splits off of $MU$ $p$-locally and there is a sequence 
\[ S =T(0)\to T(1) \to \dots \to T(n)\to T(n+1)\to \dots \to T(\infty)=BP\]
of homotopy commutative ring spectra.  Again, we refer to the maps $T(n)\to BP$ in this sequence as the canonical maps throughout. 
It is not known, however, whether this sequence rigidifies to a sequence of maps of $E_2$ ring spectra. We therefore include this as an assumption in our work and discuss this issue in more detail in Section \ref{T(n)}. 
\begin{thm}\label{main thm intro} Let $p$ be a fixed prime. 
\begin{enumerate}
\item[(a)] The Tate-valued Frobenius maps 
\[ \varphi_p \colon \thinspace THH(X(k)) \to THH(X(k))^{tC_p}\]
are equivalences after $p$-adic completion for each integer $k\ge 1$.
\item[(b)] Assuming each of the canonical maps $T(n)\to BP$ are maps of $E_2$ ring spectra, the Tate-valued Frobenius maps 
\[ \varphi_p \colon \thinspace  THH(T(n)) \to THH(T(n))^{tC_p}\]
are equivalences after $p$-adic completion for each $ n\ge 0$. 
\end{enumerate}
\end{thm}
The cases $k=1$, $n=0$, and $k=n=\infty$ were already known and we do not shed new light on these results. In fact, the case $k=1$ and $n=0$ is exactly the Segal conjecture for the sphere spectrum and the group $C_p$. The cases $k=\infty$ and $n=\infty$ are the Segal conjecture for topological Hochschild homology $MU$ and $BP$ respectively and they were proven by Lun{\o}e-Nielsen--Rognes \cite{LNR11}. 

As a consequence of Theorem \ref{main thm intro} and Tsalidis' theorem \cite{Tsa98,BBLNR07}, there are equivalences 
\begin{align}\label{Tsalidis cor} 
THH(R)^{C_{p^k}}\overset{\Gamma}{\lra} THH(R)^{hC_{p^k}} 
\end{align}
after $p$-adic completion when $R$ is $X(k)$ and $k\ge 1$ is an integer and $p$ any prime. Under our running assumptions on $T(n)$, we also prove that \eqref{Tsalidis cor} is an equivalence after $p$-adic completion when $R=T(n)$ 
 for each integer $n\ge 0$ and prime $p$. 
We define 
\[ TC^{-}(R;\mathbb{Z}_p):=\left ( THH(R)^{h\mathbb{T}} \right )^{\wedge}_p,\]
where $(-)^{\wedge}_p$ denotes $p$-adic completion and 
\[ TP(R;\mathbb{Z}_p):=\left ( THH(R)^{t\mathbb{T}} \right )^{\wedge}_p.\]
When $R$ is a connective $E_1$ ring spectrum, there is an equivalence
\begin{align}
\label{eq: equivalence} TP(R;\mathbb{Z}_p)\overset{\simeq}{\longrightarrow} (THH(R)^{tC_p})^{h\mathbb{T}}
\end{align}
by \cite[Lem. II.4.2]{NS18} and we abuse notation and write 
\[ \varphi_p \colon \thinspace TC^{-}(R;\mathbb{Z}_p)\longrightarrow  TP(R;\mathbb{Z}_p) \]
for the composite of $\varphi_p^{h\mathbb{T}}$ the with the inverse of the equivalence \eqref{eq: equivalence} in the homotopy category. As a consequence of Theorem \ref{main thm intro}, we therefore have the following corollary. 
\begin{cor}\label{main cor}
Let $p$ be a fixed prime. 
\begin{enumerate}
\item[(a)] There are equivalences 
\[ \varphi_p\colon \thinspace TC^{-}(X(k);\mathbb{Z}_p) \simeq   TP(X(k);\mathbb{Z}_p) \]
for each integer $k\ge 1$.
\item[(b)] Assuming the canonical maps $T(n)\to BP$ are maps of $E_2$ ring spectra for each $n$, then there are equivalences
\[ \varphi_p \colon \thinspace TC^{-}(T(n);\mathbb{Z}_p)\simeq TP(T(n);\mathbb{Z}_p)\] 
for each integer $n\ge 0$.
\end{enumerate}
\end{cor}

Following \cite{NS18}, we let 
\[ \text{can} \colon \thinspace TC^{-}(R)\to TP(R;\mathbb{Z}_p);\]
denote the map appearing in the isotropy separation diagram (see \cite[Eq. 35]{AR02} where this map is called $R^h$) post-composed with the map from $TP(R)$ to its $p$-adic completion.
The equalizer of the diagram 
\[  \xymatrix{   TC^{-}(R) \ar@<.5ex>[rr]^(.4){(\varphi_p)_{p\in \mathbb{P}}} \ar@<-.5ex>[rr]_(.4){(\text{can})_{p\in \mathbb{P}}} && \prod_{p\in \mathbb{P}}TP(R;\mathbb{Z}_p) }\]
in the homotopy category of spectra is equivalent to topological cyclic homology $TC(R)$ of $R$, when $R$ is connective by Nikolaus--Scholze \cite[Cor. 1.5]{NS18}. Corollary \ref{main cor} implies that $TC(X(n))$ is the fiber of the difference of the canonical map and an equivalence. Again, we emphasize that knowing that the maps $\varphi_p$ induce isomorphisms insufficiently high degrees for $R=H\mathbb{F}_p$ and $R=\ell$, where $\ell$ denotes the $p$ complete Adams summand is a crucial step in the calculations of $TC(H\mathbb{F}_p)$ and $TC(\ell)$ by Hesselholt--Madsen and Ausoni--Rognes respectively \cite{HM97,AR02}. 

By the Dundas--Goodwillie--McCarthy theorem \cite[Thm. 7.2.2.1]{DGM13}, there is a pullback 
\[ 
	\xymatrix{
	K(X(k))\ar[r]\ar[d]_{\text{tr}} & K(\mathbb{Z})\ar[d]^{\text{tr}} \\
	TC(X(k))\ar[r] & TC(\mathbb{Z}),
	}
\]
induced by the linearization map $X(k)\to H\mathbb{Z}$ and naturality of the cyclotomic trace map for each $k\ge 1$. 
In fact, it is also known that there are isomorphism 
\[K_i(X(k)_p;\mathbb{Z}_p)\cong TC_i(X(k)_p;\mathbb{Z}_p)\cong  TC_i(X(k);\mathbb{Z}_p)\]
by \cite[Thm 7.3.1.8 ]{DGM13} and \cite{Dun97}. Our computation is therefore a key step in the direction of computing the algebraic K-theory groups $K_i(X(k))$.

The program of Dundas--Rognes \cite[Sec. 4.5]{DR18} suggests that algebraic K-theory of the sphere spectrum can be computed from the algebraic K-theory of $X(k)$ by descent along the map of ring spectra $S\to X(k)$. In particular, the Amitsur complex associated to the map $S\to X(k)$ produces a cosimplicial resolution of $S$ by smash powers of $X(k)$. In \cite{DR18}, Dundas--Rognes prove that algebraic K-theory satisfies cosimplicial descent and therefore one can approach the computation of the algebraic K-theory groups of the sphere spectrum using the algebraic K-theory of the spectra $X(k)$. Our results therefore also provide a first step towards computing algebraic K-theory of the sphere spectrum. Just as the the computation of the stable homotopy groups of spheres is one of the most fundamental problems in algebraic topology, the algebraic K-theory of the sphere spectrum is one of the most fundamental questions in algebraic K-theory because of its applications to stable diffeomorphisms of manifolds; see for example Waldhausen \cite{Wal87} for a survey. 

Besides approximating algebraic K-theory of the sphere spectrum, computing invariants of the spectra $X(n)$ is useful in its own right because of the connection to formal groups. Just as $MU$ carries the universal formal group law, the spectra $X(n)$ carry the universal formal $n$-bud. Therefore, any spectrum $E$ equipped with a formal $n$-bud receives a map of homotopy commutative ring spectra from $X(n)$. If this map of homotopy commutative ring spectra can be rigidified to an actual $E_1$ ring map, all of the invariants $THH(E)$, $TP(E;\mathbb{Z}_p)$, $TC^{-}(E;\mathbb{Z}_p)$, 
$TC(E;\mathbb{Z}_p)$ and $K(E;\mathbb{Z}_p)$ we have discussed are modules over the corresponding invariant for $X(n)$. The analogous statements for $T(n)$ are true where formal $n$-buds are replaced with $p$-typical formal $n$-buds under the additional assumption that the canonical maps $T(n)\to BP$ are maps of $E_2$ ring spectra. This relationship to formal $n$-buds is also one of the reasons the spectra $X(k)$ and $T(n)$ played such a vital role in the proof of the Nilpotence theorem by Devinatz--Hopkins--Smith \cite{DHS88}.  

\subsection{Outline} 
In Section \ref{sec X(k)}, we recall the definitions of the spectra $X(k)$ and  $T(n)$ and their homology. In Section \ref{bok sec}, we compute homology of topological Hochschild homology of $X(k)$ and $T(n)$ using the B\"okstedt spectral sequence. In Section \ref{Greenlees section} we compute the continuous homology of the spectrum $THH(X(n))^{tC_p}$, where continuous homology is defined as in Equation \eqref{continuous homology}. In Section \ref{SectionSinger}, we relate $H_*(THH(X(n))^{tC_p})$ to the homological Singer construction $R_+(H^c_*(THH(X(n))))$ \cite{AGM85}. We then use this to prove our main theorem, Theorem \ref{main thm}, and subsequent corollaries, Corollary \ref{cor 1} and Corollary \ref{cor 2}. In Section 6, we discuss the same computations for $T(n)$ under the assumption that the canonical map $T(n)\to BP$ is a map of $E_2$ ring spectra. We also show that the obstruction to our assumption holding can be described in terms of an explicit class in an Atiyah-Hirzebruch spectral sequence.

\subsection{Conventions}
Throughout, homology is always taken with coefficients in $\mathbb{F}_p$ and $H_*(-,\mathbb{F}_p)$ will simply be denoted $H_*(-)$. We let $HH_*(R)$ denote the Hochschild homology over $\mathbb{F}_p$ of a graded $\mathbb{F}_p$ algebra $R$ and we write $\otimes$ for the tensor product over $\mathbb{F}_p$.  We use the terminology $p$-adic equivalence as shorthand for a map that induces an equivalence after $p$-adic completion.  We write $P(x_1,\dots x_k)$ for a polynomial algebra with generators $x_1,\dots, x_k$ over $\mathbb{F}_p$ and $E(y_1,\dots,y_k)$ for an exterior algebra with generators $y_1,\dots y_k$ over $\mathbb{F}_p$. We will use the convention of Milnor \cite{Mil58} and write 
$\ca_*= P(\zeta_1, \zeta_2, \dots ) $
for the dual Steenrod algebra at the prime $2$ where we define $\zeta_i^2:=\overline{\xi}_i$. We will write 
$\ca_*=P(\overline{\xi}_1, \overline{\xi}_2, \dots )\otimes E(\tau_0, \tau_1, \dots ) $
for the Steenrod algebra at primes $p\ge 3$. We will write 
$\mathbb{T} \subset \mathbb{C}$ for the circle group, regarded as the unit vectors in the complex numbers, and let $C_{n} \subseteq \t$ be subgroup of $n$-th roots of unity. We will use the notation $\dot{=}$ for an equality that holds only up to multiplication by a unit in $\mathbb{F}_p$. In fact, our results only rely on the homotopy category of spectra in most cases and all of our results could be proven in any of the standard models for the closed symmetric monoidal stable model category of spectra discussed in \cite{MMSS01}. 

\subsection{Acknowledgments} The authors would like to thank Mark Behrens and Andrew Salch for their comments on earlier versions of this paper and an anonymous referee helpful comments. The second author was partially supported by NSF grant DMS-1547292.

\section{The spectra $X(k)$ and $T(n)$}\label{sec X(k)}
The Ravenel spectra  $X(k)$ and $T(n)$ were first defined in \cite{Rav84}, where Ravenel proposed his family of well regarded conjectures known as the Ravenel conjectures. The spectra $X(k)$ and $T(n)$ were vital in the proofs of the Ravenel conjectures by Devinatz--Hopkins--Smith \cite{DHS88,HS98}. The spectrum $X(k)$ is constructed as the Thom spectrum $X(k):= Th(f)$ 
of the $2$-fold loop map  $f\co \Omega SU(k)\lra \Omega SU\simeq BU.$
By Lewis \cite[Ch. IX]{LMS86}, the Thom spectrum of an $m$-fold loop map is an $E_m$ ring spectrum and hence $X(k)$ is an $E_2$ ring spectrum. Functoriality of the Thom construction gives a sequence 
\[S = X(1) \to X(2) \to \cdots \to X(\infty) = MU\]
of maps of $E_2$ ring spectra. We refer to the maps $X(k)\to MU$ as the canonical maps.

Just as $MU$ splits $p$-locally as a wedge of suspensions of $BP$, $X(k)$ splits $p$-locally as a wedge of suspensions of $T(n)$ where $n$ is chosen such that $p^n \leq k < p^{n+1}$ \cite[Thm. 6.5.1]{Rav86}. By \cite[pg. 16]{Hop84}, the spectra $T(n)$ fit together to form a sequence whose homotopy colimit is $BP$. We refer to the map from $T(n)$ to the homotopy colimit $BP$ as the canonical map.
 By \cite{BM13}, we know that $BP$ is an $E_4$ ring spectrum and by \cite[Thm. 6.5.1]{Rav86} we know that $T(n)$ is a homotopy commutative homotopy associative ring spectrum for all $n \geq 0$. However, it is not known whether $T(n)$ is an $E_2$ ring spectrum or that the canonical map $T(n)\to BP$ is a map of $E_2$ ring spectra for any $n>0$. We discuss the multiplicative structure of $T(n)$ in more detail in Section \ref{T(n)}. 

The homology of $MU$ and $BP$ are well-known and the homology of $X(k)$ and $T(n)$ can be found in \cite[Sec. 3]{Rav84}. We recall these results in the following omnibus lemma. 

\begin{lem} Let $p$ be a fixed prime. 
\begin{enumerate}
\item[(a)]
The homology of $MU$ is determined by the isomorphism 
$H_*(MU) \cong P(b_1,b_2,b_3,\ldots)$
of $\mathcal{A}_*$ comodule algebras where $|b_i| = 2i$. 
The map $X(k)\to MU$ induces the isomorphism
\[
H_*(X(k)) \cong P(b_1,\ldots,b_k).
\]
of $\mathcal{A}_*$ comodule algebras onto its image in $H_*(MU)$. 
\item[(b)]
The map $BP\to H\mathbb{F}_p$ induces an isomorphism
$H_*(BP) \cong P(\overline{\xi}_1,\overline{\xi}_2,\overline{\xi}_3,\ldots)$
of $\mathcal{A}_*$-comodule algebras \cite{Rav84} onto its image in $\mathcal{A}_*$,
where $|\overline{\xi}_i| =2p^i-2$ for $i\ge 1$. The canonical map $T(n)\to BP$ induces an isomorphism
\[ H_*(T(n)) \cong 
P(\overline{\xi}_1,\overline{\xi}_2,\ldots,\overline{\xi}_n)
\]
of $\mathcal{A}_*$ comodule algebras onto the image of $H_*(T(n))$ in $H_*(BP)$. 
\end{enumerate}
\end{lem}

\section{Homology of the topological Hochschild homology of $X(n)$} \label{bok sec}
Our first step is to calculate the homology of $THH(X(n))$. For this, we use the B\"okstedt spectral sequence \cite{Bok86}. We quickly review the construction of this spectral sequence following \cite[Ch. IX]{EKMM07}. Let $R$ be an $S$-algebra. Then $THH(R)$ is the geometric realization of the cyclic bar construction on $R$, which is a simplicial ring spectrum $B_{\bullet}^{\text{cy}}(R)$ with $k$-simplices $R^{\wedge k+1}$ and the standard cyclic face and degeneracy maps. 
The B\"okstedt spectral sequence arises from applying homology to the skeletal filtration of the simplicial spectrum $B_{\bullet}^{\text{cy}}(R)$ and it is a strongly convergent spectral sequence, with signature
\begin{align}\label{bokstedt ss}
E^2_{*,*}(R) \cong HH_*(H_*(R))\Rightarrow H_*(THH(R)).
\end{align}

\begin{prop}\label{prop: homology of thh of xn} 
The homology of $THH(X(n))$ is 
\[
H_*(THH(X(n)) \cong P(b_1,b_2,\ldots,b_n) \otimes E(\sigma b_1, \sigma b_2, \ldots, \sigma b_n). 
\]
\end{prop}
\begin{proof} 
To compute the $E^2$-term of the B\"okstedt spectral sequence \eqref{bokstedt ss} for $X(k)$, we first need to compute the Hochschild homology $HH_*(H_*(X(k)))$. The Hochschild homology of grade polynomial algebras with generators in even degree can be computed using the Koszul resolution, which shows that
\[ HH_*(P(x_1,x_2, \ldots))\cong P(x_1,x_2, \ldots)\otimes E(\sigma x_1, \sigma x_2 ,\ldots), \]
see \cite[Prop. 2. 1]{MS93} for example. 
Therefore
\[E_{**}^2(X(k)) = P(b_1,b_2,\ldots, b_n) \otimes E(\sigma b_1,\sigma b_1,\ldots,\sigma b_k)\]
where  $|b_i|=(0,2i)$ and $|\sigma b_i|=(1, 2i)$. 

Since there are no generators in filtration degree greater than one and the $d_r$ differentials shift filtration degree by $r\ge 2$, there are no possible differentials and the spectral sequence collapses at the $E^2$-term. We may resolve multiplicative extensions in the same way as in the proof of \cite[Lem 6.2]{LNR11}. At $p=3$, we simply observe that the square of an odd degree generator must be trivial in a graded commutative $\mathbb{F}_p$ algebra. At $p=2$ since $THH(X(n))$ is an $E_1$ ring spectrum, we still have a Dyer--Lashofoperation $Q^{|\sigma x_\ell|/2}(\sigma x_\ell)=(\sigma x_\ell)^2$ in $H_*(THH(X(n)))$, see \cite[Ch. III Thm. 3.2]{BMMS86} and \cite[Thm. 5.2]{Law19}. In particular, we know 
\[ (\sigma x_\ell)^2=Q^{|\sigma x_\ell|/2}(\sigma x_\ell)=\sigma Q^{|\sigma x_\ell|/2}(x_{\ell})=0\] 
since $Q^{|\sigma x_\ell|/2}(x_{\ell})=0\in H_*(X(k))$ and Dyer--Lashofoperations commute with the operator $\sigma$ by \cite{Bok86}. This resolves the remaining possible multiplicative extensions. The abutment is a graded $\mathbb{F}_p$ module so additive extensions are not possible. 
\end{proof} 

\section{Homological Tate fixed points of $THH(X(n))$} \label{Greenlees section}
The next step is to compute the continuous homology of the Tate construction of $THH(X(k))$, where continuous homology is defined as \eqref{continuous homology}. Our discussion follows the discussion of the Tate construction in \cite[Proof of Prop. 4.9]{LNR11}, and we refer the reader to that proof for additional details.

Let $X$ be a $\mathbb{T}$-spectrum regarded as a $C_p$-spectrum by restricting the $\mathbb{T}$-action to the $p$-th roots of unity in $\mathbb{T}$. 
Then the \emph{Tate construction} of $X$, denoted $X^{tC_p}$, may be modeled by
\[X^{tC_p} := \left ( \widetilde{E}\t \wedge F(E\t_+,X)\right )^{C_p}\]
where $E\t$ is the total space of the universal $\t$ bundle over $B\t$ and $\widetilde{E}\mathbb{T}$ is the cofiber of the map $E\t_+ \to S^0$ induced by the collapse map $E\t\to *$ by applying the disjoint basepoint functor $(-)_+$ from spaces to based spaces. In particular, there are $C_p$-equivariant equivalences $E\t\simeq EC_p$ and $\widetilde{E}\t\simeq \widetilde{E}C_p$. We fix a model $\widetilde{E}\mathbb{T}$ as the one point compactification of $S^{\infty\mathbb{C}}$ of the countably infinite dimensional $\mathbb{C}$-vector space equipped with a $\mathbb{T}$-action by coordinate-wise multiplication. 
We also fix a $\mathbb{T}$-CW structure on $\widetilde{E}\mathbb{T}$ with  $(2n)$-skeleton and $(2n-1)$-skeleton $S^{n\mathbb{C}}$ so that the filtration quotients 
$S^{j\mathbb{C}}/S^{(j-1)\mathbb{C}}$ in the double speed skeletal filtration of $\widetilde{E}\mathbb{T}$ are 
homeomorphic to $\Sigma^{2j-1}\mathbb{T}_+$. 
We may also equip $\widetilde{E}\mathbb{T}$ with a $C_p$-action by restriction to the subgroup of $p$-th roots of unity and equip $\widetilde{E}\mathbb{T}$ with a $C_p$-CW structure with $S^{j\mathbb{C}}$ as $2j$-skeleton such that the double speed skeletal filtration quotient $\Sigma^{2j-1}\mathbb{T}_+$ is equipped with a $C_p$ CW structure as in \cite[Def. 3.8]{LNR11}. We let $\widetilde{E}_j$ for $j\ge 0$ be the $j$-skeleton of this skeletal filtration of $\widetilde{E}\t$ as a $C_p$-CW complex. Following Greenlees \cite{Gre87}, we define $\widetilde{E}_j$ for $j<0$ to be the Spanier-Whitehead dual $D(\widetilde{E}_{-j-1})$. We can then define 
\[ X^{tC_p}[j]=\left ( \widetilde{E}\mathbb{T}/\widetilde{E}_j \wedge F(E\t_+,X) \right )^{C_p} \]
to obtain a filtration
\begin{align}\label{Greenlees filtration}
X^{tC_p} \to \cdots X^{tC_p}[j] \to X^{tC_p}[j+1] \to X^{tC_p}[j+2] \to \cdots
\end{align}
of $X^{tC_p}$, which we  refer to as the \emph{Greenlees filtration}.\footnotemark \footnotetext{This is not exactly the same as the filtration as defined by Greenlees in \cite{Gre87}, but the difference between the two just amounts to a different choice of model for $EC_p$ as a $C_p$-CW complex.}
The filtration quotients are 
\[\left ( \widetilde{E}_j/\widetilde{E}_{j-1}\wedge F(E\t_+,X) \right )^{C_p}\simeq \Sigma^jX \] 
for each integer $j$, as shown in \cite[Proof of Prop. 4.9]{LNR11}. 

Applying homology produces an exact couple whose associated spectral sequence is called the \emph{homological Tate spectral sequence}, which converges strongly to the \emph{continuous homology}
\begin{align}\label{continuous homology} 
H_*^c(X^{tC_p})=\lim_j H_*(X^{tC_p}[j])
\end{align}
when $X$ is bounded below and $H_*(X)$ is finite type by \cite[Prop. 4.1]{LNR11}.

We begin by analyzing this spectral sequence for $THH(X(n))$, which is clearly bounded below and whose homology is clearly finite type by Lemma \ref{prop: homology of thh of xn}.  The homological Tate spectral sequence is therefore a strongly convergent spectral sequence, with signature
\begin{align}\label{eq2} \widehat{E}_2^{s,t} = \widehat{H}^{-s}(C_p; H_t(THH(X(k)))\Rightarrow H^c_{s+t}(THH(X(k))^{tC_p}) 
\end{align} 
where $H^c_{s+t}(THH(X(k))^{tC_p})$ is a continuous $\mathcal{A}_*$-comodule by \cite[Prop. 4.1]{LNR12}.\footnotemark \footnotetext{See \cite[Sec. 2.2]{LNR12} for a survey of continuous $\mathcal{A}_*$ comodules.}

The geometric realization of the cyclic bar construction admits a canonical $\mathbb{T}$-action, so $THH(X(k))$ is equipped with a canonical $\mathbb{T}$-action. The $C_p$-action on $H_*(THH(X(k)))$ is the restriction of this $\mathbb{T}$-action, so it acts trivially on $H_*(THH(X(k)))$ because $\t$ is path connected and the action of $\t$ on $H_*(THH(X(k)))$ is continuous. Therefore, the $E_2$-term of \eqref{eq2} splits as a tensor product
\[\widehat{E}_2^{**} = \widehat{H}^{-*}(C_p;H_*(THH(X(k)))) \cong \widehat{H}^{-*}(C_p;\mathbb{F}_p) \otimes H_*(THH(X(k)).\] 
Since $\widehat{H}^{-*}(C_2;\mathbb{F}_2) \cong P(t^{\pm 1})$ with $|t|=-1$ and $\widehat{H}^{-*}(C_p;\mathbb{F}_p) \cong E(h)\otimes P(t^{\pm 1})$ with $|h|=-1$ and $|t|=-2$ if $p>2$, we can identify the $E^2$-term of \eqref{eq2}: 
\[ \widehat{E}_2^{**}\cong
\begin{cases}
P(t^{\pm 1}) \otimes P(b_1,b_2,\ldots,b_n) \otimes E(\sigma b_1,\sigma b_2,\ldots,\sigma b_k), \quad &p=2, \\
E(h) \otimes P(t^{\pm1}) \otimes P(b_1,b_2,\ldots,b_k)\otimes E(\sigma b_1,\sigma b_2,\ldots, \sigma b_n), \quad &p>2.
\end{cases}
\] 
For $p=2$, the degrees of the generators are $|t| = (-1,0)$, $|b_i| = (0,2i)$, and $|\sigma b_i| = (0,2i+1)$. For $p>2$, the degrees of the generators are $|h| = (-1,0)$, $|t| = (-2,0)$, and the degrees of $b_i$ and $\sigma b_i$ are the same as in the case $p=2$. 

We note that the operator $\sigma(-)$ is induced in homology by the natural map 
\[ S^1\wedge R \to S^1_+\wedge R\to S^1_+\wedge THH(R)\to THH(R)\]
where the final map in the composite is exactly the action of $S^1_+$.  Thus, for each integer $k\ge 1$, the induced map 
\[ H_*(S^1)\otimes H_*(X(k))\cong H_*(S^1\wedge X(k))\to H_*(THH(X(k)))\]
sends $\iota \otimes b_i$ to $\sigma b_i$ for all $1\le i\le k$, where $\iota$ is the fundamental class in $H_1(S^1)$ \cite[Prop. 3. 2]{MS93}. This fact will be used in the following result. 

\begin{prop}\label{prop Einfty page} 
In the homological Tate spectral sequence \eqref{eq2}, the $d^2$-differentials are generated by 
\[
d_2(b_i) \thinspace \dot{=} 
\begin{cases}
\thinspace t^2 \sigma b_i & \text{ for } p=2, \text{ and }\\ 
   \thinspace t\sigma b_i  & \text{ for } p>2,
  \end{cases}
\]
via the Leibniz rule for all $1\le i\le k$. 
\end{prop} 
\begin{proof} 
We compare the spectral sequence \eqref{eq2} for $k$ finite with the the case $k=\infty$ where $X(\infty)=MU$ by examining the map of spectral sequences induced by the canonical map $X(k)\to MU$. The $d^2$-differentials for the Tate spectral sequence converging to $H^c_*(THH(MU)^{tC_p})$ were computed in \cite[Prop. 6.3]{LNR11}. They showed that for all $i \geq 1$, one has $d^2(b_i) = t^2 \sigma b_i$ for $p=2$, and $d^2(b_i) \dot{=} t \sigma b_i$ for $p>2$. This can be proven directly by lifting the differentials from the Tate spectral sequence converging to the continuous homology $H^c_*(THH(X(k))^{t\t})$ of the Tate construction along the inclusion $C_p \to \t$, where continuous homology is defined as in \cite[Prop. 7.1]{BR05}. The $\t$-Tate spectral sequence differentials arise from looking at the skeletal filtration of the model of $E\t$ given by $S(\infty \mathbb{C})$ and noting that the attaching maps are given by the $\t$-action \cite{BR05}. 

The map $X(k) \to MU$ induces an injective map of $E^2$-terms of homological Tate spectral sequences. In particular, any differential $d^2(b_i) = t^2 \sigma b_i$ for $p=2$ or $d^2(b_i) \dot{=} t \sigma b_i$ for $p>2$ in the $MU$ case must also occur in the $X(k)$ case when $i \leq k$ and when $i>k$ both the source and target of the differential in the target spectral sequence are not in the image of the map of $E^2$-terms. This gives the stated $d^2$-differentials. Therefore, the map of $E^3$-terms of homological Tate spectral sequences is again injective. Since the homological Tate spectral sequence for $MU$ collapses at the $E^3$-term by \cite[Prop. 6.3]{LNR11} and the map of $E^3$-terms is injective the homological Tate spectral sequence also collapses for $X(k)$ for any integer $k\ge 1$. 
\end{proof}

The continuous homology $H^c_*(THH(X(k))^{tC_p})$ of the Tate construction of $THH(X(k))$ follows from the above pattern of differentials. 
\begin{cor}\label{cohoxa} Let $k\ge 1$ be an integer. 
The continuous homology of the Tate construction of $THH(X(k))$ is determined by an isomorphism of continuous $\mathcal{A}_*$-comoudle algebras
\[H^c_*(THH(X(k))^{tC_p}) \cong 
\begin{cases}
P(t^{\pm1}) \otimes P(b^2_1,\ldots,b^2_k) \otimes E(b_1 \sigma b_1,\ldots,b_k\sigma b_k),  \quad & p=2, \\
E(h) \otimes P(t^{\pm 1}) \otimes P(b^p_1,\ldots,b^p_k) \otimes E(b^{p-1}_1\sigma b_1,\ldots,b^{p-1}_k \sigma b_k), \quad &p>2,
\end{cases}
\]
where the continuous $\mathcal{A}_*$-coaction is determined by the inclusion into $H^c_*(THH(MU)^{tC_p})$.
\end{cor}
\begin{proof}
The $E_{\infty}$-term in the spectral sequence is determined by Proposition \ref{prop Einfty page}. To solve extensions, we note that the map of spectral sequences induced by $X(k) \to MU$ is multiplicative by functoriality of the homological Tate spectral sequence since the map $X(k)\to MU$ is a map of $E_2$ ring spectra. Therefore, the continuous $\mathcal{A}_*$-comodule extensions and multiplicative extensions follow from \cite[Props. 6.3, 6.4]{LNR11} and naturality of the Tate-valued Frobenius map. 
\end{proof}

\section{Identification with the Singer construction}\label{SectionSinger}
The goal of this section is to prove the Segal Conjecture for $THH(X(n))$, Theorem \ref{main thm}. Our proof proceeds by modifying the proof of the Segal Conjecture for $THH(MU)$ given by Lun{\o}e-Nielsen--Rognes in \cite{LNR11}. To avoid repeating some of their technical arguments and constructions, we include precise references to their paper where possible. Note that we cannot directly apply the proofs of in Lun{\o}e-Nielsen--Rognes in \cite{LNR11} because they use results from \cite{BR05} in the case of $MU$ case that rely on having an $E_{\infty}$-ring spectrum structure on $MU$. No such $E_{\infty}$-ring spectrum structure exists on $X(n)$, which can be seen using Dyer--Lashoff operations $Q_1(\overline{\xi}_i) =\overline{\xi}_{i+1}$ computed by Steinberger \cite{BMMS86}. 

By the isotropy separation diagram for topological Hochschild homology and the cyclotomic structure on $THH(R)$ (see \cite[Prop. 4.1]{HM97}), the map $THH(R)^{C_p} \overset{\Gamma}{\lra} THH(R)^{hC_p}$ is a $p$-adic equivalence whenever the Tate-valued Frobenius map
\[ \varphi_p \colon \thinspace THH(R) \lra THH(R)^{tC_p} \]
is an $p$-adic equivalence. We will show this by exhibiting an $Ext$-equivalence (Definition \ref{Ext-equiv})
\begin{align}\label{eq3} H_*(THH(R)) \overset{\epsilon_*}{\lra} R_+(H_*(THH(R))) \overset{\Phi_n}{\lra} H^c_*(THH(R)^{tC_p}), \end{align}
where $R_+(-)$ is the homological Singer construction \cite[Def. 3.7]{LNR12}. 
The $p$-adic equivalence then follows from comparing Adams spectral sequence to the inverse limit Adams spectral sequence as in \cite{LDMA80,AGM85,LNR12}. The map $\epsilon_*$ is an $Ext$-equivalence by \cite[Prop. 1.2, Thm. 1.3]{AGM85}, so we must show that the map $\Phi_n$ is an $Ext$-equivalence. 

\begin{defin}\label{Ext-equiv}
A homomorphism $M\to N$ of $\ca_*$-comodules is an \emph{$Ext$-equivalence} if the induced homomorphism
\[Ext^{s,t}_{\ca_*}(\mathbb{F}_p,M) \to Ext^{s,t}_{\ca_*}(\mathbb{F}_p,N)\]
is an isomorphism for all $s \geq 0$ and $t\in \mathbb{Z}$.
\end{defin}

We begin with the case $R = X(k)$ and $p>2$. By Corollary \ref{cohoxa}, the continuous homology of the Tate construction on $THH(X(k))$ is determined by the isomorphism
\[ H_*^c(THH(X(k))^{tC_p})\cong E(h) \otimes P(t^{\pm 1}) \otimes P(b_1^p, b_2^p,\ldots, b_k^p)\otimes E(b_1^{p-1}\sigma b_1 ,b_2^{p-1}\sigma b_2\ldots,b_k^{p-1}\sigma b_n). \]
On the other hand, one can consider the homological Tate spectral sequence, with signature
\[ \widehat{E}^{**}_2 = \widehat{H}^{-*} (C_p; H_*(THH(X(k))^{\wedge p})) \Rightarrow  H_*(((THH(X(k))^{\wedge p})^{tC_p}) \cong R_+(H_*(THH(X(k)))) \]
where the isomorphism on the right-hand side follows from \cite[Thm. 5.9]{LNR12}. In this case, the homological Singer construction can be expressed as
\[ R_+(H_*(THH(X(k)))) = E(h) \otimes P(t^{\pm 1}) \otimes P(b_1^{\otimes p},b_2^{\otimes p},\ldots,b_k^{\otimes p}) \otimes E(\sigma b_1^{\otimes p},\sigma b_2^{\otimes p},\ldots, \sigma b_k^{\otimes p}) \]
which is in bijection with $H^c_*(THH(X(k))^{tC_p})$ via $b_i^p \mapsto b_i^{\otimes p}$ and $b_i^{p-1} \sigma b_i \mapsto t^m \otimes \sigma b_i^{\otimes p}$ where $m = (p-1)/2$. The goal is to promote this filtration-shifting bijection to an isomorphism of complete $\ca_*$-comodules. 

The homology of $X(k)$ and $THH(X(k))$ are sub-$\ca_*$-comodules of the homology of $MU$ and $THH(MU)$, respectively. Consequently, $H^c_*(THH(X(k))^{tC_p})$ is a complete sub-$\ca_*$-comodule of $H^c_*(THH(MU)^{tC_p})$. Therefore, the formulas and computations leading up to \cite[Props. 7.2, 7.3]{LNR11} carry over mutatis mutandis.  
In particular, we obtain maps 
\begin{align}
\label{f map} R_+(H_*(X(k))) \otimes_{H_*(X(k))} H_*(THH(X(k))) \overset{f}{\lra} R_+(H_*(THH(X(k))) \\
\label{g map} 	R_+(H_*(X(k))) \otimes_{H_*(X(k))} H_*(THH(X(k))) \overset{g}{\lra} H^c_*(THH(X(k))^{tC_p}) 
\end{align}
defined by $f(\alpha \otimes \beta) = R_+(\eta_*)(\alpha) \cdot \epsilon_*(\beta)$  and $g(\alpha \otimes \beta) = \eta^t_*(\alpha) \cdot \hat{\Gamma}_*(\beta)$, where $R_+(\eta_*)$, $\epsilon_*$, $\eta^t_*$ and $\hat{\Gamma}_*$ are the $H_*(X(k))$-linear maps 
\begin{align*}
R_+(\eta_*)\co R_+(H_*(X(k)))\rightarrow R_+(H_*(THH(X(k))) \\
\epsilon_*\co H_*(THH(X(k))) \rightarrow R_+(H_*(THH(X(k)) \\
\eta^t_*\co R_+(H_*(X(k))\rightarrow H_*^c(THH(X(k))^{tC_p}) \\
\hat{\Gamma}_*\co H_*(THH(X(k)))\rightarrow H_*^c(THH(X(k))^{tC_p}) 
\end{align*}
induced by the usual unit map 
\[ \eta\co X(k)\to THH(X(k)),\] 
Tate diagonal 
\[\epsilon\co THH(X(k))\rightarrow (THH(X(k))^{\wedge p})^{tC_p},\] 
and the Tate-valued Frobenius map 
\[ \varphi_p \co THH(X(k))\rightarrow THH(X(k))^{tC_p}.\] 

There are filtrations  of the above $\ca_*$-comodules defined by
\[F^kH_*(X^{tC_p}) = im\left (H_*^c(X^{tC_p}) \to H_*(X^{tC_p}[j])\right )\]
where $H_*(X^{tC_p}[j])$ is the homology of the $j$-th term in the Greenlees filtration \eqref{Greenlees filtration}.
In particular, this defines a filtration on $R_+(H_*(R))$ and $R_+(H_*THH(R))$ for an $S$-algebra $R$ which is bounded below and finite type because, by \cite[Thm. 5.9]{LNR12}, there are isomorphisms $R_+(H_*R)\cong H_*^c ( \left ( R^{\wedge p} \right )^{tC_p})$ and 
$ R_+(H_*THH(R))\cong H_*^c ( \left ( THH(R)^{\wedge p} \right )^{tC_p}).$\footnote{In fact, this result has recently been extended by Nikolaus--Scholze \cite[Theorem III.1.7]{NS18} who show that the map $X\to (X^{\wedge p})^{tC_p}$ exhibits $(X^{\wedge p})^{tC_p}$ as the $p$-completion of $X$ for all bounded below spectra without the finite type hypothesis. } 
The maps $f$ and $g$ defined in \eqref{f map} and \eqref{g map} induce maps $\{f_j\}$ and $\{g_j\}$ of cofiltered $\ca_*$-comodules.

\begin{prop} The maps $\{f_j\}$ and $\{g_j\}$ of cofiltered $\ca_*$-comodules are strict maps of cofiltered $\ca_*$-comodules which assemble into pro-isomorphisms whose limits $\hat{f}$ and $\hat{g}$ are isomorphisms of complete $\ca_*$-comodules. 
\end{prop}

\begin{proof} Our proof is modified from the proof of \cite[Prop. 7.2]{LNR11}. We will only provide the proof for $\{f_j\}$ since the proof for $\{g_j\}$ is similar. In each total degree $d$, $f_j$ defines a map
\[f_{j,d}\co [F^jR_+(H_*(X(k)))\otimes E(\sigma b_1,\sigma b_2,\ldots,\sigma b_n)]_d \to F^jR_+(H_*(THH(X(k)))_d.\]
For each $k$, we would like to define compatible maps
\[\phi_{j,d}\co [F^NR_+(H_*(THH(X(k))))]_d \to [F^jR_+(H_*(X(k))) \otimes E(\sigma b_1,\sigma b_2,\ldots,\sigma b_k)]_d\]
with $N = N(j,d) = p(j-d)+d$, such that the composition $\phi_{j,d} \circ f_{N,d}$ is equal to the structural surjection 
\[ \xymatrix{ [F^NR_+(H_*(X(k)))\otimes E(\sigma b_1,\sigma b_2\ldots,\sigma b_k)]_d \ar@{->>}[r] &  [F^jR_+(H_*(X(k)))\otimes E(\sigma b_1,\sigma b_2 \ldots,\sigma b_k)]_d  } \]
and such that the composition $f_{j,d} \circ \phi_{j,d,}$ is equal to the structural surjection 
\[ \xymatrix{ [F^NR_+(H_*(THH(X(k)))]_d \ar@{->>}[r] & [F^jR_+(H_*(THH(X(k))))]_d. }\] 
We can then conclude that the collection $\{f_{j,d}\}_j$ forms a pro-isomorphism with pro-inverse $\{\phi_{j,d}\}_j$ in each total degree $d$. 
These maps therefore assemble into a pro-isomorphism $\{f_j\}$ with pro-inverse $\{\phi_j\}$. 

In \cite[Proof of Thm. 7.2]{LNR11}, Lun{\o}e-Nielsen--Rognes decompose the group
\[[R_+(H_*(MU)) \otimes E(\epsilon_*(\sigma m_\ell)| \ell \geq 1)]_d\] 
into a direct sum indexed by strictly increasing sequences 
$L = (\ell_1 < \cdots < \ell_r)$ of natural numbers of length $r \geq 0$. 
Then they define the maps $\phi_{j,d}$ for $MU$ using this decomposition on \cite[Pg. 618-619]{LNR11}. 
The desired maps $\phi_{j,d}$ for $X(k)$ follow from exactly the same steps. 
To decompose $[R_+(H_*(X(k))) \otimes E(\epsilon_*(\sigma b_1), \ldots, \epsilon_*(\sigma b_k)))]_d$ into a direct sum, 
we restrict to strictly increasing sequences $L = (\ell_1 < \cdots < \ell_r)$ where $0 \leq r \leq n$. 
Using the notation $\epsilon_L =\epsilon_*(\sigma b_{\ell_1})\cdot \ldots \cdot \epsilon_*(\sigma b_{\ell_r})$, we obtain homomorphisms 
\[[F^{N-s_L}R_+(H_*(X(k))) \otimes \mathbb{F}_p\{ \epsilon_L\}]_d \to [F^j R_+(H_*(X(k)))\otimes E(\sigma b_1,\sigma b_2,\ldots,\sigma b_k)]_d\]
defined by  $\epsilon_L \mapsto \sigma b_{\ell_1} \ldots \sigma b_{\ell_r}$ where $s_L=-(p-1)(2{\ell_1}+\dots +2{\ell_r}+r)$. 
The definition of $\phi_{j,d}$ is then completed by taking the direct sum over $L$.
The filtration shift estimates from \cite[Pg. 618-619]{LNR11} carry over to the $X(k)$ case. Therefore the maps $f_{j,d}$ and $\phi_{j,d}$ compose into the desired structural surjections in each total degree $d$. 
This implies that the set $\{f_j\}$ is a strict map of inverse systems which assembles into a pro-isomorphism. 

The analogous pro-isomorphisms can be defined for $X(k)$ when $p=2$ by essentially the same argument. 
Setting $\Phi_{X} := \hat{g} \circ \hat{f}^{-1}$ for the corresponding $\hat{g}=\text{lim}g_j$ and $\hat{f}=\text{lim}f_j$ yields the desired isomorphism.
\end{proof}
\begin{cor}\label{caci} There is an isomorphism of complete $\ca_*$-comodules
\[ \Phi_{X(k)} \co R_+(H_*(THH(X(k)))) \lra H^c_*(THH(X(k))^{tC_p}) .\]
\end{cor}

Therefore, when $R=X(k)$, the composite map \eqref{eq3} is an $Ext$-equivalence for each $k\ge 1$. We therefore have proven the main theorem. 

\begin{thm}\label{main thm}
The Tate-valued Frobenius map
\[\varphi_p \colon \thinspace THH(X(k)) \lra THH(X(n))^{tC_p}\]
is a $p$-adic equivalence for all integers $k\ge 1$ and primes $p$.
\end{thm}

\begin{proof}
The Tate-valued Frobenius map
\[ \varphi_p \co THH(X(k)) \to THH(X(k))^{tC_p}\]
induces a map from the Adams spectral sequence, with signature
\begin{align}\label{ASS1} E^{s,t}_2 = Ext^{s,t}_{\ca_*}(\mathbb{F}_p,H_*(THH(X(k)))) \Rightarrow \pi_{t-s}(THH(X(k)))
\end{align}
to the strongly convergent inverse limit Adams spectral sequence, with signature 
\begin{align}\label{ASS2} E^{s,t}_2 = Ext^{s,t}_{\ca_*}(\mathbb{F}_p,H^c_*(THH(X(n))^{tC_p})) \Rightarrow \pi_{t-s}(THH(X(n))^{tC_p}\end{align}
	of \cite[Prop. 2.2]{LNR12}. 

By \cite[Thm. 5.9]{LNR12}, there is an isomorphism 
\begin{align} \label{Singer const identification}
	H_* \left (THH(X(n))^{\wedge p}\right )^{tC_p}\cong R_+(H_*THH(X(n)))
\end{align}
and by
\cite[Proposition 1.2, Theorem 1.3]{AGM85} the map 
\[H_*(THH(X(n)))\to R_+(H_*THH(X(n)))\] 
induces an $Ext$-equivalence. Since we have proven that there is an isomorphism of complete $\mathcal{A}_*$-comodules
\begin{align}\label{key iso}
R_+(H_*THH(X(n)))\to H_*THH(X(n))^{tC_p},
\end{align}
there is an isomorphism between the $E_2$-terms of \eqref{ASS1} and \eqref{ASS2}.  The result then follows from strong convergence of the inverse limit of Adams spectral sequence.
\end{proof}

By the isotropy separation diagram (see \cite[Prop. 4.1]{HM97}) as well as Tsalidis' theorem \cite{Tsa98,BBLNR07}, we conclude the following corollary. 
\begin{cor}\label{cor 1}
Let $k\ge 1$ and $m\ge 0$ be integers and fix a prime $p$. 
The map
\[ THH(X(k))^{C_{p^m}} \overset{\Gamma}{\lra} THH(X(k))^{hC_{p^m}}\]
is a $p$-adic equivalence.
\end{cor}

The map $\varphi_p \colon \thinspace THH(X(k)\to THH(X(k))^{tC_p}$ an $\mathbb{T}$ equivariant map because $THH(X(k))$ is cyclotomic in the sense of \cite{NS18}. Theorem \ref{main thm} therefore implies that the Tate-valued Frobenius map is an equivalence of Borel $\mathbb{T}$-equivariant spectra in the sense of \cite{NS18}, by \cite[Thm. 5.14.]{Joy08},  Consequently, we have the following corollary. 
\begin{cor}\label{cor 2}
The Tate-valued Frobenius map induces an equivalence
\[ \varphi_p \colon \thinspace TC^{-}(X(k);\mathbb{Z}_p)\simeq  TP(X(k);\mathbb{Z}_p)\]
for all primes $p$ and all integers $k\ge 1$. 
\end{cor}

\section{The Segal conjecture for topological Hochschild homology of $T(n)$}\label{T(n)}
We conclude by briefly describing analogous results for the spectrum $T(n)$ under the assumptions that $T(n)$ is an $E_2$ ring spectrum and that the canonical maps $T(n)\to BP$ are maps of $E_2$ ring spectra for each $n$. 
These assumptions are necessary in our argument because they imply $THH(T(n))$ is a ring spectrum and the map of strongly convergent homological Tate spectral sequences
\[
\xymatrix{
	\widehat{H}^{-*}(C_p;H_*(THH(T(n))))\ar[r] \ar@{=>}[d] &  \widehat{H}^{-*}(C_p;H_*(THH(BP)))\ar@{=>}[d] \\
	H_*^c(THH(T(n))^{tC_p})\ar[r] & H_*^c(THH(BP)^{tC_p})
}
\]
 is multiplicative. The assumption that the canonical maps $T(n)\to BP$ are $E_2$ ring spectrum maps can be possibly be weakened to knowing either that the canonical maps $T(n)\to BP$ are maps of $E_1$ ring spectra or that $T(n)$ is an $E_2$ ring spectrum for each $n>0$, but neither of theses results are known. 

We begin by discussing the plausibility of our assumption that the canonical maps $T(n)\to BP$ are $E_2$ ring spectrum maps for each $n\ge 0$. We can express $T(n)$ as the colimit
\[T(n)=\underset{\epsilon_k}{\colim}\thinspace X(k),\]
in the homotopy category where $\epsilon_k$ is the restriction of the Quillen idempotent $\epsilon \co  MU\to MU$ to $X(k)$ as defined in \cite[Lem 1.3.5]{Hop84} and $k$ satisfies $p^n\le k <p^{n-1}$. In particular, this implies that $T(n)$ is homotopy commutative and homotopy associative, see \cite[Thm. 6.5.1]{Rav86}. Work of Chadwick--Mandell ~\cite{CM15} shows that $\epsilon \co MU\to MU$ is a map of $E_2$ ring spectra and the map $MU\to BP$ is a map of $E_2$ ring spectra. To prove that $T(n)$ is an $E_2$ ring spectrum and the map $T(n)\to BP$ is a map of $E_2$ ring spectra, it would suffice to show that the map $\epsilon_k \co X(k)\to X(k)$ is a map of $E_2$ ring spectra for all $k\ge 1$. 
It  would follow that $T(n)$ is an $E_2$ ring spectrum because the colimit in $E_2$ ring spectra is computed as the colimit of underlying spectra. Since the idempotent $\epsilon_k$ is known to be compatible with the idempotent $\epsilon$, in the sense that for each $k$ there is a commutative diagram 
\[ 
	\xymatrix{
	X(k)\ar[d]_{\epsilon_k} \ar[r] & MU \ar[d]^{\epsilon} \\
	X(k) \ar[r] & MU, \\
	}
\]
it would also imply, by naturality of the colimit in the category of $E_2$ ring spectra, that the canonical maps
\[ T(n)\to BP \]
are maps of $E_2$ ring spectra.

Fix $k$ and $n$ such that $p^{n-1}\le k<p^n$ throughout this section. 
Let $\Ring(X,Y)$ denote the set of homotopy classes of maps of ring spectra $X \to Y$ in the stable homotopy category, and let $\eRing(A,B)$ be the space of $E_2$ ring maps $A \to B$. We would like to show that $\epsilon_k\in \Ring(X(k),X(k))$  pulls back to a class in $\pi_0(\eRing(X(k),X(k)))$ along the map 
\begin{align}\label{EquationStructure}
 \pi_0(\eRing(X(k),X(k))) \to \Ring(X(k),X(k)). 
\end{align}
By using methods from \cite[Sec. 6]{CM15}, we can identify 
\[ \pi_0(\eRing(X(k),X(k)))\cong \widetilde{sl}_1X(k)^2(BSU(k)) \] and 
\[ \Ring(X(k),X(k))\cong \widetilde{sl}_1X(k)^0(\mathbb{C}P^{k-1}).\] 
The map (\ref{EquationStructure}) is induced by the map 
\[ \Sigma^2 \mathbb{C}P^{k-1} \to \Sigma^2\Omega SU(k) \to B^2\Omega SU(k)\simeq BSU(k).\]
Therefore it suffices to examine the map of Atiyah-Hirzebruch spectral sequences, which is given on $E_2$ pages by
\[ H^{s+2}(BSU(k); \pi_{-t} \widetilde{sl}_1X(k)) \to H^{s}(\mathbb{C}P^{k-1}; \pi_{-t} \widetilde{sl}_1X(k)). \]
We can understand this map with integral coefficients
\[  H^{*+2}(BSU(k);\mathbb{Z})\to H^{s}(\mathbb{C}P^{k-1};\mathbb{Z}) \]
where $H^{*+2}(BSU(k);\mathbb{Z})\cong P(x_2,x_3,\ldots ,x_k)$ with $|x_i|=2i$ and $H^*(\mathbb{C}P^{k-1};\mathbb{Z})\cong P_k(u)$ with $|u|=2$.  By \cite[Prop. 6.3]{CM15}, the element $x_i$ maps to $(-1)^iu^i$ for all $1\le i\le k$ and all decomposables map to zero. Therefore, it suffices to understand the map of Atiyah-Hirzebruch spectral sequences modulo decomposables. We may determine the class $z$ detecting $\epsilon_k$ in the target spectral sequence and a class $\tilde{z}$ mapping to it from the source spectral sequence. Our goal then is to show that $\tilde{z}$ is a permanent cycle. 
 
In the work of Chadwick--Mandell, the class analogous  to $\tilde{z}$ is a permanent cycle for bidegree reasons; there are no possible targets for differentials because all of the spectra they consider have homotopy groups concentrated in even degrees. Since the homotopy groups $\pi_*(X(k))$ are not known to be concentrated in even degrees, we cannot rule out the possibility of $\tilde{z}$ supporting a long differential. We make the following assumption, which implies that $T(n)$ is an $E_2$ ring spectrum for each $n\ge 0$ and the canonical maps $T(n)\to BP$ are $E_2$ ring spectrum maps. 

\begin{assump}
The class $\tilde{z}$ is a permanent cycle in the Atiyah-Hirzebruch spectral sequence with abutment $\widetilde{sl}_1X(k)^*(BSU(k))$.
\end{assump}

Assuming that the canonical maps $T(n) \to BP$ are maps of $E_2$ ring spectra, we can prove the Segal Conjecture for $THH(T(n))$ by following the same strategy as we did for $THH(X(k))$. The homology $H_*(THH(T(n)))$ can be computed using the B{\"o}kstedt spectral sequence \eqref{bokstedt ss}. In this case, the B\"okstedt spectral sequence again collapses at the $E^2$ term because of the bidegrees of the algebra generators and the multiplcitive extensions may again be resolved using the $E_1$ Dyer--Lashof algebra. This part only uses the assumption that $T(n)$ is an $E_2$ ring spectrum and not the assumption that the canonical maps $T(n)\to BP$ are maps of $E_2$ ring spectra. One obtains
\[ H_*(THH(T(n)))\cong P(\overline{\xi}_1, \overline{\xi}_2,\ldots,\overline{\xi}_n)\otimes E(\sigma \overline{\xi}_1, \sigma \overline{\xi}_2,\ldots,\sigma \overline{\xi}_n). \] 

The continuous homology $H^c_*(THH(T(n))^{tC_p})$ can be computed using the homological Tate spectral sequence by comparison with the homological Tate spectral sequence converging to $H^c_*(THH(BP)^{tC_p})$ which was computed in ~\cite[Prop. 6.8]{LNR11}. Notably, the homological Tate spectral sequence for $BP$ still collapses at the $E^3$ term and the map of $E^3$-terms is still injective. The map of spectral sequences is also multiplicative under our assumptions, which allows us to determine multiplicative extensions in the source spectral sequence. This is where we use our assumption that the map $T(n)\to BP$ is a map of $E_2$ ring spectra.
Thus, there are isomorphisms
\[ H_*^c(THH(T(n))^{tC_p})\cong 
\begin{cases}
P(t^{\pm 1}) \otimes P(\overline{\xi}^2_1,\ldots,\overline{\xi}^2_n) \otimes E(\overline{\xi}_1 \sigma \overline{\xi}_1,\ldots,\overline{\xi}_n\sigma \overline{\xi}_n) \quad& \text{for }p=2,\\
E(h) \otimes P(t^{\pm 1}) \otimes P(\overline{\xi}_1^p,\ldots, \overline{\xi}_n^p)\otimes E(\overline{\xi}_1^{p-1}\sigma \overline{\xi}_1 ,\ldots,\overline{\xi}_n^{p-1}\sigma \overline{\xi}_n) \quad& \text{for } p>2
\end{cases}
\]
of continuous $\mathcal{A}_*$-comodule algebras.
For $p=2$, the degrees of the generators are $|t| = (-1,0)$, $|\overline{\xi}_i| = (0,2^{i+1}-2)$, and $|\sigma \overline{\xi}_i| = (0,2^{i+1}-1)$. For $p>2$, the degrees of the generators are $|h| = (-1,0)$, $|t| = (-2,0)$, $|\overline{\xi}_i| = (0,2p^i-2)$, and $|\sigma \overline{\xi}_i| = (0,2p^i-1)$. 

The homological Singer construction $R_+(H_*(THH(T(n))))$ can be identified using \cite[Thm. 5.9]{LNR12} as in \eqref{Singer const identification}. The identification with the homological Singer construction 
$$R_+(H_*(THH(T(n)))) \cong H_*^c(THH(T(n))^{tC_p})$$
as continuous $\mathcal{A}_*$ comodules
follows from a modification of the proof of \cite[Thm. 7.2]{LNR11}. This modification is similar to the modification proving the analogous isomorphism \eqref{key iso} for $THH(X(n))$ in Section \ref{SectionSinger}. One concludes using the strongly convergent inverse limit Adams spectral sequence that the Tate-valued Frobenius map
\[THH(T(n)) \overset{\Gamma}{\to} THH(T(n))^{tC_p}\]
is a $p$-adic equivalence.  The analogues of Corollary \ref{cor 1} and Corollary \ref{cor 2} also hold for $T(n)$ under our running assumptions for each $n\ge 0$ .



\end{document}